\documentclass[11pt]{article}%
\usepackage{amsmath}
\usepackage{amsfonts}
\usepackage{amssymb}
\usepackage{graphicx}%
\newtheorem{theorem}{Theorem}

\newtheorem{conjecture}[theorem]{Conjecture}
\newtheorem{corollary}[theorem]{Corollary}

\newtheorem{lemma}[theorem]{Lemma}

\newtheorem{problem}[theorem]{Problem}
\newtheorem{proposition}[theorem]{Proposition}
\newtheorem{remark}[theorem]{Remark}

\newenvironment{proof}[1][Proof]{\noindent\textbf{#1.} }{\ \rule{0.5em}{0.5em}}
\graphicspath{    {converted_graphics/}    {/}}
\begin{document}

\begin{center}
{\LARGE Extinction and the Allee Effect in an Age Structured}

{\LARGE Ricker Population Model with Inter-stage Interaction}
\end{center}

\medskip

\begin{center}
N. LAZARYAN and H. SEDAGHAT
\end{center}

\medskip

\begin{abstract}
We study the evolution in discrete time of certain age-structured populations,
such as adults and juveniles, with a Ricker fitness function. We determine
conditions for the convergence of orbits to the origin (extinction) in the
presence of the Allee effect and time-dependent vital rates. We show that when
stages interact, they may survive in the absence of interior fixed points, a
surprising situation that is impossible without inter-stage interactions. We
also examine the shift in the interior Allee equilibrium caused by the
occurrence of interactions between stages and find that the extinction or
Allee threshold does not extend to the new boundaries set by the shift in
equilibrium, i.e. no interior equilibria are on the extinction threshold.

\end{abstract}

\medskip

\section{Introduction}

The evolution of certain types of biological populations from a period, or
time interval, $n$ to the next may be modeled by the discrete system
\begin{align}
x_{n+1}  &  =s_{n}x_{n}+s_{n}^{\prime}y_{n}\label{1}\\
y_{n+1}  &  =x_{n}^{\lambda}e^{r_{n}-bx_{n+1}-cx_{n}} \label{2}%
\end{align}
where $\lambda,c>0$, $b\geq0$ with $s_{n}\in\lbrack0,1)$, $s_{n}^{\prime}%
\in(0,1]$ and $r_{n}\in(-\infty,\infty)$ for all $n$.

A common example is the population of a single species whose members are
differentiated by their age group, where, e.g. $x_{n}$ and $y_{n}$ represent,
respectively the population densities of adults and juveniles in time period
$n$. In this setting, $s_{n}$ and $s_{n}^{\prime}$ denote the survival rates
of adults and juveniles, respectively. For examples of stage-structured
models, see \cite{cush98} and \cite{LS} and references thereof.

The time dependent parameters $r_{n},s_{n},s_{n}^{\prime}$ may be periodic in
the presence of periodic factors such as seasonal variations in the
environment, migration, harvesting, predation, etc. The effects of inter-stage
(adult-juvenile) interactions may be included with $b>0$. In this case,
Equation (\ref{2}) indicates that the juvenile density in each period is
adversely affected by adults present in the same period. Causes include
competition with adults for scarce resources like food or in some cases,
cannibalization of juveniles by adults.

The system (\ref{1})-(\ref{2}) is not in the standard form that can be
represented by a planar map. Its standard form is obtained by substituting
$s_{n}x_{n}+s_{n}^{\prime}y_{n}$ from (\ref{1}) for $x_{n+1}$ in (\ref{2}) and
rearranging terms to obtain the standard planar system%
\begin{align}
x_{n+1}  &  =s_{n}x_{n}+s_{n}^{\prime}y_{n}\label{1b}\\
y_{n+1}  &  =x_{n}^{\lambda}e^{r_{n}-(c+bs_{n})x_{n}-bs_{n}^{\prime}y_{n}}
\label{2b}%
\end{align}

In this form, the system is a special case of the age-structured model
\begin{subequations}
\label{c}%
\begin{align*}
A(t+1) &  =s_{1}(t)\sigma_{1}(c_{11}(t)J(t),c_{12}(t)A(t))J(t)+s_{2}%
(t)\sigma_{2}(c_{21}(t)J(t),c_{22}(t)A(t))A(t)\\
J(t+1) &  =b(t)\phi(c_{1}(t)J(t),c_{2}(t)A(t))A(t)
\end{align*}
introduced in \cite{Cush}. In this general case, $A(t)$ and $J(t)$ are
population densities of adults and juveniles respectively, remaining after $t$
periods. The function $\phi$ is exponential in (\ref{2b}) but other choices
may be considered for modeling different types of population dynamics
(\cite{AJ}, \cite{Ber}, \cite{LS}, \cite{LP}, \cite{Z}).

The Allee effect describes the positive correlation between population density
and its per capita birth rate. The greater the size of the population, the
better it fares. The increase in the overall fitness of the population at
greater densities is attributed to cooperation (\cite{Cour}). The Allee
principle was first introduced by W. Allee (\cite{A1}, \cite{A2}) at the time
when the prevailing focus was on the effects of overcrowding and competition
on the survival of the species. The Allee principle focuses on how low
population density, or under-crowding, affects the survival or extinction of
the species.

A distinction is made between the weak and strong Allee effects. The effect is
weak if per capita population growth is low but positive at lower densities
compared to that at higher densities. In the presence of the strong Allee
effect, population rate below a critical threshold is negative (\cite{Cour}).

Mathematically, the map that defines the dynamical system that exhibits a
strong Allee effect is characterized by three fixed points - the extinction or
zero fixed point; a small positive fixed point referred to as the Allee
threshold; and a bigger positive fixed point called the carrying capacity
(\cite{luis}). When the population size is at or above the Allee threshold,
growth in population density is observed, whereas beneath the threshold,
population density declines. When the population size is at or above the Allee
threshold, per capita growth in the population is positive, whereas beneath
the threshold, population density declines. For more details on the Allee
effect and its various contributing factors see \cite{Ber}, \cite{Cour},
\cite{cushing2}, \cite{jang}, \cite{livad}, \cite{luis}, \cite{schreiber} and
references thereof; in particular, see \cite{cushing}, \cite{cushing1},
\cite{elaydi}, \cite{lidicker}, \cite{yakubu}.

We study the system in (\ref{1})-(\ref{2}) by first folding it into the
second-order scalar equation (\ref{3}) below. The strong Allee effect is
exhibited when $\lambda>1$, as might be expected. However, when $b>0$ the
details of the Allee effect such as the nature of the extinction region and
its boundary are not fully understood in the case of (\ref{3}). In particular,
in the autonomous case the extinction region is smaller than expected if
$b>0$. In this paper we establish this fact for a special case of (\ref{3})
and also obtain general conditions for the convergence of solutions of
(\ref{3}) to zero when $\lambda>1$ and $b\geq0.$

The main results of this paper are as follows: Theorem \ref{bc0} and its
immediate corollary state that when $\lambda>1$ extinction occurs for all
values of the system parameters $b$, $c$, $r_{n}$, $s_{n}$, $s_{n}^{\prime}$
if the initial values $x_{0},y_{0}$ are suitably restricted. Alternatively,
extinction occurs for \textit{all} non-negative initial values if the system
parameters are sufficiently restricted. While these results are in line with
what is known in the literature about Ricker-type systems, Lemma \ref{eoz} and
Theorem \ref{nfpa} and their immediate corollaries contain results that are
not as predictable. They establish for the case $s_{n}=0$ for all $n$ (e.g. a
semelparous species) that the population may become extinct or alternatively,
its size may oscillate depending on whether orbits enter, or avoid certain
regions of the positive quadrant. In particular, we find that if the
interaction parameter $b$ is positive then survival may occur for open regions
of the parameter space even if the system contains no positive fixed points;
this is rather unexpected (and false if $b=0$).

\section{Convergence to zero: general conditions}

The standard planar form (\ref{1b})-(\ref{2b}) has additional time-dependent
parameters in the exponential function (\ref{2b}) that did not exist in the
original system. For this reason and others seen below, we find it more
convenient to study (\ref{1})-(\ref{2}) using the alternative folding method
discussed in \cite{sed}. The system (\ref{1})-(\ref{2}) may be folded into a
scalar second-order difference equation by first solving (\ref{1}) for $y_{n}$
to obtain:%

\end{subequations}
\begin{equation}
y_{n}=\frac{1}{s_{n}^{\prime}}(x_{n+1}-s_{n}x_{n}) \label{yn}%
\end{equation}

Next, back-shifting the indices in (\ref{2}) and substituting the result in
(\ref{1}) yields%
\[
x_{n+1}=s_{n}x_{n}+s_{n}^{\prime}x_{n-1}^{\lambda}e^{r_{n-1}-bx_{n}-cx_{n-1}%
}=s_{n}x_{n}+x_{n-1}^{\lambda}e^{r_{n-1}+\ln s_{n}^{\prime}-bx_{n}-cx_{n-1}}%
\]
or equivalently,%
\begin{equation}
x_{n+1}=s_{n}x_{n}+x_{n-1}^{\lambda}e^{a_{n}-bx_{n}-cx_{n-1}}\label{2a}%
\end{equation}
where $a_{n}=r_{n-1}+\ln s_{n}^{\prime}$. Note that this equation does not
introduce additional time-dependent parameters in the exponential function.

In terms of populations of adults and juveniles, starting from initial adult
and juvenile population densities $x_{0}$ and $y_{0}$ respectively, a solution
$\{x_{n}\}$ of the scalar equation (\ref{2a}) yields the adult population
density. The juvenile population density $y_{n}$ is found via (\ref{yn}). The
initial values for (\ref{2a}) are $x_{0}$ and $x_{1}=s_{0}x_{0}+s_{0}^{\prime
}y_{0}.$

Without loss of generality, we may assume that $c=1$ and normalize the
equation by a simple change of variables and parameters: $x_{n}\rightarrow
cx_{n}$, $a_{n}\rightarrow a_{n}+(\lambda-1)\ln c$ and $b\rightarrow b/c$.
Thus, we obtain the following more convenient form of (\ref{2a})
\begin{equation}
x_{n+1}=s_{n}x_{n}+x_{n-1}^{\lambda}e^{a_{n}-bx_{n}-x_{n-1}} \label{3}%
\end{equation}

The next result gives general sufficient conditions for the boundedness of
solutions and for their convergence to 0; also see Lemma \ref{eoz} below for
another general result on convergence to 0.

\begin{theorem}
\label{bc0}Assume that $\lambda>0$, $s\doteq\sup_{n\geq0}\{s_{n}\}<1$ and
$A\doteq\sup_{n\geq0}\{a_{n}\}<\infty$.

(a) Every non-negative solution of (\ref{3}) is eventually uniformly bounded.

(b) Let $\lambda>1$ and define%
\[
\rho=\exp\left(  -\frac{A-\ln(1-s)}{\lambda-1}\right)  .
\]

If $\{x_{n}\}$ is a solution of (\ref{3}) with $x_{1},x_{0}<\rho$ then
$\lim_{n\rightarrow\infty}x_{n}=0$.

(c) Let $\lambda>1$. If
\begin{equation}
\;A<\ln(1-s)+(\lambda-1)[1-\ln(\lambda-1)] \label{b}%
\end{equation}
then every non-negative solution of (\ref{3}) converges to $0$.
\end{theorem}

\begin{proof}
(a) If $\{x_{n}\}$ is a solution of (\ref{3}) with non-negative initial values
then $x_{n}\geq0$ so that $\{x_{n}\}$ is bounded below by 0. Further,
\[
x_{n+1}\leq sx_{n}+x_{n-1}^{\lambda}e^{A-x_{n-1}}\leq sx_{n}+e^{A}\left[
x_{n-1}e^{-(1/\lambda)x_{n-1}}\right]  ^{\lambda}%
\]

For all $r>0$ the maximum value of the function $xe^{-rx}$ is $1/re$ which
occurs at the unique critical value $1/r.$ Thus, for all $n\geq0$
\[
x_{n+1}\leq sx_{n}+e^{A}\left(  \frac{\lambda}{e}\right)  ^{\lambda}%
=sx_{n}+\lambda^{\lambda}e^{A-\lambda}%
\]

This inequality yields%
\begin{align*}
x_{1}  &  \leq sx_{0}+\lambda^{\lambda}e^{A-\lambda}\\
x_{2}  &  \leq sx_{1}+\lambda^{\lambda}e^{A-\lambda}\leq s^{2}x_{0}%
+\lambda^{\lambda}e^{A-\lambda}(1+s)\\
&  \vdots\\
x_{n}  &  \leq s^{n}x_{0}+\lambda^{\lambda}e^{A-\lambda}(1+s+\cdots
+s^{n-1})=\frac{\lambda^{\lambda}e^{A-\lambda}}{1-s}+\left(  x_{0}%
-\frac{\lambda^{\lambda}e^{A-\lambda}}{1-s}\right)  s^{n}%
\end{align*}

Since the second term above vanishes as $n$ goes to infinity it follows that
the solution $\{x_{n}\}$ is eventually uniformly bounded, e.g. by the number
$1+\lambda^{\lambda}e^{A-\lambda}/(1-s)$ for all sufficiently large $n$.

(b) Assume that $\lambda>1$ and choose initial values $x_{0},x_{1}<\rho.$ Then
there is $\varepsilon>0$ such that
\[
\mu\doteq\max\{x_{0},x_{1}\}\leq\exp\left(  -\frac{\varepsilon+A-\ln
(1-s)}{\lambda-1}\right)  \doteq\rho_{\varepsilon}<\rho
\]

Thus, $e^{A}x_{0}^{\lambda-1}\leq e^{-\varepsilon+\ln(1-s)}=e^{-\varepsilon
}(1-s)$ and it follows that%
\[
x_{2}\leq sx_{1}+x_{0}^{\lambda}e^{A}\leq sx_{1}+\left(  e^{A}x_{0}%
^{\lambda-1}\right)  x_{0}\leq sx_{1}+e^{-\varepsilon}(1-s)x_{0}\leq\lbrack
s+e^{-\varepsilon}(1-s)]\mu
\]

Notice that if $\delta\doteq s+e^{-\varepsilon}(1-s)$ then $\delta<1$ and
thus, $x_{2}\leq\delta\mu<\rho_{\varepsilon}.$ Next, since $x_{1}\leq
\rho_{\varepsilon}$%
\[
x_{3}\leq sx_{2}+x_{1}^{\lambda}e^{A}\leq sx_{2}+\left(  e^{A}x_{1}%
^{\lambda-1}\right)  x_{1}\leq sx_{2}+e^{-\varepsilon}(1-s)x_{1}\leq\delta
\max\{x_{1},x_{2}\}\leq\delta\mu
\]
where the last inequality is true because $x_{1}\leq\mu$ and $x_{2}\leq
\delta\mu.$ Next, since $x_{2},x_{3}<\rho_{\varepsilon}$ it follows that
$e^{A}x_{2}^{\lambda-1},e^{A}x_{3}^{\lambda-1}<e^{-\varepsilon}(1-s)$ and
thus,
\begin{align*}
x_{4}  &  \leq sx_{3}+x_{2}^{\lambda}e^{A}\leq sx_{3}+\left(  e^{A}%
x_{2}^{\lambda-1}\right)  x_{2}<sx_{3}+e^{-\varepsilon}(1-s)x_{2}\leq
\delta\max\{x_{2},x_{3}\}\leq\delta^{2}\mu\\
x_{5}  &  \leq sx_{4}+x_{3}^{\lambda}e^{A}\leq sx_{4}+\left(  e^{A}%
x_{3}^{\lambda-1}\right)  x_{3}<sx_{4}+e^{-\varepsilon}(1-s)x_{3}\leq
\delta\max\{x_{3},x_{4}\}\leq\delta^{2}\mu
\end{align*}

We have thus shown that
\[
x_{0},x_{1}\leq\mu<\rho_{\varepsilon},\quad x_{2},x_{3}\leq\delta\mu
<\rho_{\varepsilon},\quad x_{4},x_{5}\leq\delta^{2}\mu<\rho_{\varepsilon}.
\]

Proceeding the same way, it follows inductively that%
\[
x_{2n},x_{2n+1}\leq\delta^{n}\mu
\]
for all $n\geq0.$ Therefore, $\lim_{n\rightarrow\infty}x_{n}=0.$

(c) Assume that $\lambda>1$. Since $ue^{-ru}\leq1/er$ it follows that
\[
x_{n+1}\leq sx_{n}+x_{n-1}^{\lambda}e^{A-x_{n-1}}\leq sx_{n}+e^{A}%
x_{n-1}\left(  x_{n-1}e^{-x_{n-1}/(\lambda-1)}\right)  ^{\lambda-1}\leq
sx_{n}+e^{A}\left(  \frac{\lambda-1}{e}\right)  ^{\lambda-1}x_{n-1}%
\]

If $\sigma\doteq s+e^{A-\lambda+1}(\lambda-1)^{\lambda-1}$ then $\sigma<1$ by
(\ref{b}) and we have shown above that%
\begin{equation}
x_{n+1}\leq\sigma\max\{x_{n},x_{n-1}\} \label{maxn}%
\end{equation}
for all $n\geq0.$ Now, for every pair of initial values $x_{0},x_{1}\geq0$,
(\ref{maxn}) implies that
\[
x_{2},x_{3}\leq\sigma\max\{x_{0},x_{1}\}=\sigma\mu,\quad x_{4},x_{5}\leq
\sigma\max\{x_{2},x_{3}\}\leq\sigma^{2}\mu,\ldots
\]
and by induction,%
\[
x_{2n},x_{2n+1}\leq\sigma^{n}\mu.
\]

Therefore, $\lim_{n\rightarrow\infty}x_{n}=0$ and the proof is complete.
\end{proof}

\medskip

It is worth emphasizing that Part (b) of the above theorem is valid for all
values of the system parameters $b$, $a_{n}$, $s_{n}$ if the state-space
parameters $x_{0},x_{1}$ are suitably restricted. Extinction always occurs
when $\lambda>1$ if the population sizes are sufficiently low, irrespective of
the other system parameters. On the other hand, Part (c) is valid for all
non-negative values of the state-space parameters $x_{0},x_{1}$ if the system
parameters are sufficiently restricted. In this case extinction is inevitable
no matter what the initial population sizes are.

We define the\textit{ extinction region} of the system (\ref{1})-(\ref{2}) to
be the largest subset $E$ of $[0,\infty)\times\lbrack0,\infty)$ in which
extinction occurs; i.e. if $(x_{k},y_{k})\in E$ for some $k\geq0$ then
$(x_{n},y_{n})\in E$ for $n\geq k$ and $\lim_{n\rightarrow\infty}(x_{n}%
,y_{n})=(0,0).$ By the \textit{base component} of $E$ or the \textit{Allee
region} we mean the component (maximal connected subset) $E_{0}$ that contains
the origin in its boundary. In general, $E_{0}$ is a proper subset of $E$ as
it excludes other possible components of $E$ that are separated from the
origin. However, since nonzero orbits of (\ref{1})-(\ref{2}) do not map to
zero directly, all orbits converge to the origin by passing through $E_{0}$.
Thus all components of $E$ map into $E_{0}$.

The next result is an immediate consequence of Theorem \ref{bc0} for the
system (\ref{1})-(\ref{2}). Note that if $x_{n}<\mu$ for some $\mu>0$ and all
$n$ then by (\ref{yn})%
\[
y_{n}<\frac{\mu-s_{n}x_{n}}{s_{n}^{\prime}}\leq\frac{\mu}{s_{n}^{\prime}}%
\leq\frac{\mu}{\inf_{n\geq0}s_{n}^{\prime}}.
\]

\begin{corollary}
\label{gen}Assume that $\lambda>1$, $\inf_{n\geq0}s_{n}^{\prime}\doteq
\sigma>0$ and let $A,s,\rho$ be as defined in Theorem \ref{bc0}.

(a) The rectangle $[0,\rho)\times\lbrack0,\rho/\sigma)$ is an invariant set of
the system (\ref{1})-(\ref{2}). If $\{(x_{n},y_{n})\}$ is an orbit of this
system with a point $(x_{k},y_{k})$ in the rectangle for some $k\geq0$ then
$(x_{n},y_{n})\in\lbrack0,\rho)\times\lbrack0,\rho/\sigma)$ for $n\geq k$ and
the orbit $\{(x_{n},y_{n})\}$ converges to the origin. Thus $[0,\rho
)\times\lbrack0,\rho/\sigma)\subset E_{0}$.

(b) If (\ref{b}) holds then every orbit of the system (\ref{1})-(\ref{2}) in
the positive quadrant of the plane converges to the origin; i.e. the origin is
a global attractor of all orbits so $E_{0}=[0,\infty)\times\lbrack0,\infty).$
\end{corollary}

\section{The Allee effect and extinction in the autonomous case}

To understand the role of inter-stage interaction in modifying the Allee
effect and the extinction region with minimum diversion, we assume (unless
otherwise stated) that all parameters are time independent; i.e. $a_{n}%
=a\in(-\infty,\infty)$, $s_{n}=s\in\lbrack0,1)$ and \thinspace$s_{n}^{\prime
}=s^{\prime}\in(0,1]$ are constants for all $n$. Then (\ref{3}) reduces to the
autonomous equation%
\begin{equation}
x_{n+1}=sx_{n}+x_{n-1}^{\lambda}e^{a-bx_{n}-x_{n-1}}. \label{4a}%
\end{equation}

\subsection{The fixed points}

The fixed points or equilibrium solutions of (\ref{4a}) are important to the
subsequent discussion. They are the roots of the equation%
\begin{equation}
x=sx+x^{\lambda}e^{a-(b+1)x} \label{5}%
\end{equation}

Clearly zero is a solution of (\ref{5}), representing the extinction
equilibrium in the biological context. The nonzero roots of (\ref{5}) are the
solutions of
\begin{equation}
1-s=x^{\lambda-1}e^{a-(b+1)x}\doteq h(x). \label{5a}%
\end{equation}

The derivative of $h$ is
\[
h^{\prime}(x)=\frac{e^{a-(b+1)x}[\lambda-1-(b+1)x]}{x^{2-\lambda}}%
\]

If $\lambda>1$ then $h^{\prime}$ has a unique positive zero at which $h$ is
maximized:
\[
x_{\max}=\frac{\lambda-1}{b+1}%
\]

Now
\[
h(x_{\max})=\left(  \frac{\lambda-1}{b+1}\right)  ^{\lambda-1}e^{a-(\lambda
-1)}\geq1-s
\]

if and only if
\begin{equation}
a\geq\ln(1-s)+(\lambda-1)[1+\ln(b+1)-\ln(\lambda-1)]. \label{fxp}%
\end{equation}

The next result summarizes the preceding discussion.

\begin{lemma}
\label{2fp}Assume that $\lambda>1$.

(a) If (\ref{fxp}) holds with strict inequality then (\ref{4a}) has two fixed
points $x^{\ast}$ and $\bar{x}$ that satisfy%
\begin{equation}
0<x^{\ast}<\frac{\lambda-1}{b+1}<\bar{x}. \label{6a}%
\end{equation}

(b) If (\ref{fxp}) holds with equality then (\ref{4a}) has a unique positive
fixed point
\[
x^{\ast}=x_{\max}=\frac{\lambda-1}{b+1}.
\]

(c) If (\ref{fxp}) does not hold, i.e.%
\begin{equation}
a<\ln(1-s)+(\lambda-1)[1+\ln(b+1)-\ln(\lambda-1)] \label{nfxp}%
\end{equation}
then (\ref{4a}) has no positive fixed points.
\end{lemma}

Note that if $x^{\ast},\bar{x}$ are fixed points of (\ref{4a}) then the fixed
points of the system (\ref{1})-(\ref{2}) are obtained using (\ref{yn}) as%
\[
\left(  x^{\ast},\frac{1-s}{s^{\prime}}x^{\ast}\right)  ,\quad\left(  \bar
{x},\frac{1-s}{s^{\prime}}\bar{x}\right)  .
\]

We refer to the first of the above fixed points as well as the $x^{\ast}$
itself as the \textit{Allee fixed point or equilibrium}. The next result
proves the interesting fact that the value of this fixed point increases (it
moves away from the origin) as the interaction parameter $b$ increases.

\begin{lemma}
\label{bux}If $\lambda>1$ then the Allee fixed point $x^{\ast}$ is an
increasing function of $b$ and its minimum value (with all other parameter
values fixed) occurs at $b=0$ as long as (\ref{fxp}) holds with strict inequality.
\end{lemma}

\begin{proof}
Since $x^{\ast}$ satisfies (\ref{5a}) taking the logarithm yields%
\[
\ln(1-s)=(\lambda-1)\ln x^{\ast}+a-(b+1)x^{\ast}%
\]

Thinking of this equation as defining $x^{\ast}$ as a function of $b,$ we take
the derivative with respect to $b$ to find that%
\[
0=\frac{\lambda-1}{x^{\ast}}\frac{dx^{\ast}}{db}-x^{\ast}-(b+1)\frac{dx^{\ast
}}{db}%
\]
which yields%
\[
\frac{dx^{\ast}}{db}=\frac{\left(  x^{\ast}\right)  ^{2}}{\lambda
-1-(b+1)x^{\ast}}%
\]

This equality and (\ref{6a}) imply that $dx^{\ast}/db>0$ and the proof is complete.
\end{proof}

\medskip

We close this section with a discussion of local stability of the origin and
the Allee fixed point $x^{\ast}.$ Let%
\[
F(x,y)=sx+y^{\lambda}e^{a-bx-y}%
\]

For each fixed point of (\ref{4a}) the eigenvalues of the linearization are
the roots of the characteristic equation%
\[
u^{2}-F_{x}u-F_{y}=0
\]
where%
\begin{equation}
F_{x}=s-by^{\lambda}e^{a-bx-y},\qquad F_{y}=(\lambda-y)y^{\lambda-1}%
e^{a-bx-y}\label{6b}%
\end{equation}

\begin{lemma}
\label{xst}Assume that $\lambda>1$ in (\ref{4a}).

(a) The origin is locally asymptotically stable;

(b) If (\ref{fxp}) holds then the positive fixed point $x^{\ast}$ is unstable;

(c) If (\ref{fxp}) holds and $x^{\ast}>s/b(1-s)$ (e.g. if $s=0$) then
$x^{\ast}$ is a repelling node.
\end{lemma}

\begin{proof}
(a) By (\ref{6b}), $F_{x}(0,0)=s$ and $F_{y}(0,0)=0$ so the roots $0,s$
of the characteristic equation $u^{2}-su=0$ are in the interval 
$(-1,1)$.

(b) By (\ref{5a}) and (\ref{6b}) the characteristic equation of the
linearization at $(x^{\ast},x^{\ast})$ is
\[
u^{2}-[s-(1-s)bx^{\ast}]u-(1-s)(\lambda-x^{\ast})=0
\]
whose roots, or eigenvalues are%
\begin{equation}
\eta^{\pm}(x^{\ast})=\frac{s-(1-s)bx^{\ast}\pm\sqrt{(s-(1-s)bx^{\ast}%
)^{2}+4(1-s)(\lambda-x^{\ast})}}{2}\label{lpm}%
\end{equation}

Note that $\eta^{\pm}(x^{\ast})$ are both real because (\ref{6a}) implies
$\lambda-x^{\ast}>1+bx^{\ast}$ so the expression under the square root in
$\eta^{\pm}(x^{\ast})$ satisfies
\begin{align*}
\lbrack s-(1-s)bx^{\ast}]^{2}+4(1-s)(\lambda-x^{\ast})  & >[1-(1-s)(1+bx^{\ast
})]^{2}+4(1-s)(1+bx^{\ast})\\
& =1+2(1-s)(1+bx^{\ast})+(1-s)^{2}(1+bx^{\ast})^{2}\\
& =[1+(1-s)(1+bx^{\ast})]^{2}%
\end{align*}
which is non-negative. Further, this calculation also shows that
\begin{equation}
\eta^{+}(x^{\ast})>\frac{s-(1-s)bx^{\ast}+1+(1-s)(1+bx^{\ast})}{2}%
=1\label{lg1}%
\end{equation}
so $x^{\ast}$ is unstable. 

(c) From the calculations above we also have the following%
\[
\eta^{-}(x^{\ast})<\frac{s-(1-s)bx^{\ast}-[1+(1-s)(1+bx^{\ast})]}%
{2}=-(1-s)(1+bx^{\ast})
\]

Therefore, $\eta^{-}(x^{\ast})<-1$ if $(1-s)(1+bx^{\ast})>1$ which may be
written as $x^{\ast}>s/b(1-s)$ to complete the proof.
\end{proof}

\medskip

\subsection{The Allee effect without inter-stage interaction}

Before examining the effect of inter-stage interaction on extinction it is
useful, for the sake of comparison, to examine the case where such interaction
does not occur. The parameters that link the two stages are $s_{n}$ and $b.$
The latter configures inter-stage interaction directly into the model while
the former does this less directly by allowing a fraction of adults to survive
into the next period and thus interact with the next generation's juveniles.
To remove all inter-stage interaction we set $s_{n}=0$ for all $n$ (e.g. the
case of a semelparous species)\ and also set $b=0.$ Therefore, (\ref{4a})
reduces to the equation%
\begin{equation}
x_{n+1}=x_{n-1}^{\lambda}e^{a-x_{n-1}}\label{3aa}%
\end{equation}

The even and odd terms of a solution of this second-order equation separately
satisfy the first-order difference equation%
\begin{equation}
u_{n+1}=u_{n}^{\lambda}e^{a-u_{n}} \label{3a}%
\end{equation}
that has been studied in some detail; see \cite{es}, \cite{elaydi}. In this
case, the population of each stage evolves separately as a single-species
population according to (\ref{3a}). Specifically, if $\{x_{n}\}$ is a solution
of (\ref{3aa}) with given initial values $x_{1},x_{0}\geq0$ then the odd terms
satisfy%
\[
x_{2n+1}=x_{2n-1}^{\lambda}e^{a-x_{2n-1}}%
\]
so $x_{2n+1}=u_{n}$ where $\{u_{n}\}$ is a solution of (\ref{3a}) with initial
value $u_{0}=x_{1}$; similarly, the even terms are $x_{2n}=u_{n}$ where
$\{u_{n}\}$ is a solution of (\ref{3a}) with initial value $u_{0}=x_{0}$.

We summarize a few of the well-known properties of (\ref{3a}) as a lemma which
we state without proof here.

\begin{lemma}
\label{ord1}Let $\lambda>1$ and define
\begin{equation}
f(u)=u^{\lambda}e^{a-u}. \label{fmap}%
\end{equation}

(a) The mapping $f$ has no positive fixed points if and only if%
\begin{equation}
\;a<(\lambda-1)[1-\ln(\lambda-1)]. \label{ac0}%
\end{equation}

(b) If%
\begin{equation}
a>(\lambda-1)[1-\ln(\lambda-1)] \label{fup}%
\end{equation}
then $f$ has positive fixed points $u^{\ast}$ and $\bar{u}$ such that
$u^{\ast}<\lambda-1<\bar{u}.$ Further, $u^{\ast}=\bar{u}=\lambda-1$ if and
only if the inequality in (\ref{fup}) is replaced with equality. We may call
$u^{\ast}$ the Allee fixed point of $f$.

(c) If
\begin{equation}
a\leq\lambda-(\lambda-1)\ln\lambda\label{ubar}%
\end{equation}
then the interval $(u^{\ast},\bar{u})$ is invariant under $f$ with $\bar
{u}<f(\lambda)\leq\lambda$. Further, $f$ is strictly increasing on this
interval with $f(u)>u$ for $u\in(u^{\ast},\bar{u})$.

(d) Assume that (\ref{fup}) holds. Then $u^{\ast}$ is unstable but $\bar{u}$
is asymptotically stable if (\ref{ubar}) holds. Further, if $\{u_{n}\}$ is a
solution of (\ref{3a}) or equivalently, of $u_{n+1}=f(u_{n})$ with $u_{0}%
\in(u^{\ast},\bar{u})$ then $\{u_{n}\}=\{f^{n}(u_{0})\}$ is increasing with
$\lim_{n\rightarrow\infty}u_{n}=\bar{u}.$

(e) Assume that (\ref{ac0}) does not hold, i.e.
\begin{equation}
a\geq(\lambda-1)[1-\ln(\lambda-1)]. \label{nac0}%
\end{equation}

Then $f(u)<u$ for $u\in(0,u^{\ast})$ and if $\{u_{n}\}=\{f^{n}(u_{0})\}$ is a
solution of (\ref{3a}) with $u_{0}<u^{\ast}$ then $\{u_{n}\}$ is decreasing
with $\lim_{n\rightarrow\infty}u_{n}=0.$

(f) If (\ref{fup}) holds and $u^{\ast}$ is the Allee fixed point of $f$ then
there is a unique fixed point $u_{\ast}>\lambda$ such that $f(u_{\ast
})=u^{\ast}$. If $u_{0}>u_{\ast}$ then $u_{n}<u^{\ast}$ for $n\geq1$ so
$\lim_{n\rightarrow\infty}f^{n}(u_{0})=0.$
\end{lemma}

\medskip

\begin{corollary}
\label{nic}In the system (\ref{1})-(\ref{2}) assume that $\lambda>1$, $b=0$,
$s_{n}=0$ for all $n$ and further, $s_{n}^{\prime}=s^{\prime}$ and $r_{n}=r$
are constants.

(a) The origin is a global attractor of all orbits $\{(x_{n},y_{n})\}$ of
(\ref{1})-(\ref{2}), i.e. $E_{0}=[0,\infty)\times\lbrack0,\infty)$ if and only
if (\ref{ac0}) holds for $a=r+\ln s^{\prime}$.

(b) Assume that $a=r+\ln s^{\prime}$ satisfies (\ref{fup}). Then $x^{\ast
}=u^{\ast}$ and $\bar{x}=\bar{u}$ where $u^{\ast},\bar{u}$ are the fixed
points of the mapping $f$ in (\ref{fmap}) and every orbit of (\ref{1}%
)-(\ref{2}) with initial point $(x_{0},y_{0})\in\lbrack0,u^{\ast}%
)\times\lbrack0,u^{\ast}/s^{\prime})$ converges to the origin, i.e.
$E_{0}\subset\lbrack0,u^{\ast})\times\lbrack0,u^{\ast}/s^{\prime}).$

(c) Assume that $a=r+\ln s^{\prime}$ satisfies (\ref{fup}) and (\ref{ubar}).
If $x_{0}\in\lbrack u^{\ast},\bar{u}]$ or$\ s^{\prime}y_{0}=x_{1}\in\lbrack
u^{\ast},\bar{u}]$ then the corresponding orbit of (\ref{1})-(\ref{2}) does
not converge to the origin. Thus $E_{0}=[0,u^{\ast})\times\lbrack0,u^{\ast
}/s^{\prime})$ in this case.

(d) If $a=r+\ln s^{\prime}$ satisfies (\ref{fup}) then the periodic sequences
\[
\{0,u^{\ast},0,u^{\ast},\ldots\},\quad\{u^{\ast},0,u^{\ast},0\ldots\}
\]
are unstable solutions of (\ref{3aa}).
\end{corollary}

\begin{proof}
(a) Both the even terms and the odd terms of every non-negative solution of
(\ref{3aa}) converge to 0 by Lemma \ref{ord1}(a). Conversely, if every orbit
of the system\ converges to the origin then (\ref{3a}) cannot have a positive
fixed point so (\ref{ac0}) holds.

(b) In this case, both the even terms and the odd terms of a solution of
(\ref{3aa}) converge to 0 by Lemma \ref{ord1}(e).

(c) Let $x_{0}\in\lbrack u^{\ast},\bar{u}].$ Then by Lemma \ref{ord1}(d)
$\lim_{n\rightarrow\infty}x_{2n}=\bar{u}$ if $x_{0}\not =u^{\ast}$ and
$x_{2n}=u^{\ast}$ for all $n$ if $x_{0}=u^{\ast}.$ In either case, the orbit
$\{(x_{n},y_{n})\}$ does not converge to the origin. A similar argument
applies if $s^{\prime}y_{0}=x_{1}\in\lbrack u^{\ast},\bar{u}].$

(d) If $x_{0}=0$ and $x_{1}=u^{\ast}$ then $x_{2n}=0$ and $x_{2n+1}=u^{\ast}$
for all $n.$ Therefore, $\{0,u^{\ast},0,u^{\ast},\ldots\}$ is a solution of
(\ref{3aa}). Next, if $x_{2n}=\varepsilon$ and $x_{2n+1}=u^{\ast}%
-\varepsilon^{\prime}$ for small $\varepsilon,\varepsilon^{\prime}>0$ then by
Lemma \ref{ord1}(e) the corresponding solution $\{x_{n}\}$ converges to 0. It
follows that $\{0,u^{\ast},0,u^{\ast},\ldots\}$ is an unstable solution. By a
similar argument, $\{u^{\ast},0,u^{\ast},0\ldots\}$ is an unstable solution of
(\ref{3aa}).
\end{proof}

\medskip

In the next section we examine the Allee effect and the extinction region when
inter-stage interaction occurs. By Lemma \ref{2fp} if $\lambda>1$ and
(\ref{fxp}) holds then (\ref{4a}) has fixed points $x^{\ast}$ and $\bar{x}$
that satisfy (\ref{6a}). By analogy with Corollary \ref{nic} it might be
conjectured that $x_{0},x_{1}<x^{\ast}$ implies extinction for (\ref{4a}).
However, we show that this is not true!

\subsection{The Allee effect with inter-stage interaction}

When $b>0$ inter-stage interactions occur because adults will be present among
juveniles. However, we keep $s=0$ so as to study the specific role of the
coefficient $b$ in modifying the Allee effect as well as simplifying some
calculations. This leaves us with the second-order equation%
\begin{equation}
x_{n+1}=x_{n-1}^{\lambda}e^{a-bx_{n}-x_{n-1}}\label{4ab}%
\end{equation}

Both (\ref{4a}) and (\ref{4ab}) display the Allee-type bistable behavior when
$\lambda>1$ as well as a range of qualitatively different dynamics depending
on the parameter values. We study some nontrivial aspects of (\ref{4ab})
related to extinction and the Allee effect. 

First, by Lemma \ref{xst}, $x^{\ast}$ is a repelling node for (\ref{4ab}) when
$b>0.$ This similarity to the one-dimensional case where there is no
age-structuring is not typical of things to come though, because the solutions
of (\ref{4ab}) originating in a small neighborhood of $x^{\ast}$ do not
converge to 0. 

To gain a better understanding of the behaviors of solutions of (\ref{4ab}),
we begin with the following basic result which is true in more general
(non-autonomous) settings.

\begin{lemma}
\label{eoz}Assume that $\lambda>1$ and let $a_{n},b_{n}$ be sequences of real
numbers such that $a=\sup_{n\geq1}a_{n}<\infty$ and $b_{n}\geq0$ for all $n.$
Further, assume that (\ref{nac0}) holds and let $u^{\ast}$ be the Allee fixed
point of the mapping $f(u)=u^{\lambda}e^{a-u}$. If $x_{k}\in(0,u^{\ast})$ for
some $k\geq0$ then the terms $x_{k},x_{k+2},x_{k+4},\ldots$ of the
corresponding solution of the equation%
\begin{equation}
x_{n+1}=x_{n-1}^{\lambda}e^{a_{n}-b_{n}x_{n}-x_{n-1}} \label{4abn}%
\end{equation}
decrease monotonically to 0. Thus, if $k$ is even (or odd) then the
even-indexed (respectively, odd-indexed) terms of the solution eventually
decrease monotonically to zero.
\end{lemma}

\begin{proof}
The inequality in (\ref{nac0}) implies that $u^{\ast}>0$ exists as a fixed
point of the mapping $f$. If $x_{k}\in(0,u^{\ast})$ for some $k\geq0$ then
$f(x_{k})<x_{k}$ by Lemma \ref{ord1} and
\begin{align*}
x_{k+2}  &  =e^{-b_{k+1}x_{k+1}}x_{k}^{\lambda}e^{a_{k+1}-x_{k}}\leq
x_{k}^{\lambda}e^{a-x_{k}}=f(x_{k})<x_{k}<u^{\ast}\\
x_{k+4}  &  =e^{-b_{k+3}x_{k+3}}x_{k+2}^{\lambda}e^{a_{k+3}-x_{k+2}}\leq
f(x_{k+2})<x_{k+2}<x_{k}<u^{\ast}%
\end{align*}
and so on. It follows that the terms $x_{k},x_{k+2},x_{k+4},\ldots$ from a
decreasing sequence with each term in $(0,u^{\ast})$. If $\inf_{j\geq
0}x_{k+2j}=\zeta>0$ then $\zeta<u^{\ast}$ so $f(\zeta)<\zeta.$ But
\[
\zeta=\lim_{j\rightarrow\infty}x_{k+2j}\leq\lim_{j\rightarrow\infty
}f(x_{k+2(j-1)})=f\left(  \lim_{j\rightarrow\infty}x_{k+2(j-1)}\right)
=f(\zeta)
\]
which not possible. Thus $\zeta=0$ and it follows that $\lim_{j\rightarrow
\infty}x_{k+2j}=0.$ The last statement of the theorem is obvious.
\end{proof}

\medskip

The next result may be compared with Corollary \ref{gen}(a) and Corollary
\ref{nic}(b),(c).

\begin{corollary}
\label{eozc}In the system (\ref{1})-(\ref{2}) assume that $\lambda>1$,
$s_{n}=0$ and $s_{n}^{\prime}=s^{\prime}>0$ for all $n$ and let $a=\sup
_{n\geq1}a_{n}<\infty.$ Also assume that (\ref{nac0}) holds and $u^{\ast}$ is
as in Lemma \ref{eoz}. If $\{(x_{n},y_{n})\}$ is an orbit of the system with
$(x_{k},y_{k})\in\lbrack0,u^{\ast})\times\lbrack0,u^{\ast}/s^{\prime})$ for
some $k$ then $\{(x_{n},y_{n})\}$ converges to the origin. In particular,
$[0,u^{\ast})\times\lbrack0,u^{\ast}/s^{\prime})\subset E_{0}$ and this
inclusion is proper if $b>0$.
\end{corollary}

\begin{proof}
The first assertion of the theorem follows readily from Lemma \ref{eoz} since
$x_{k+1}=s^{\prime}y_{k}$ by (\ref{yn}). Thus $E_{0}$ contains the rectangle
$[0,u^{\ast})\times\lbrack0,u^{\ast}/s^{\prime}).$ To see why this inclusion
is proper, we show that $E_{0}$ contains points not in the rectangle. Let
$(x_{0},y_{0})$ be the boundary point $(u^{\ast},u^{\ast}/s^{\prime})$ of the
rectangle. Then $x_{0}=u^{\ast}$ and also $x_{1}=s^{\prime}y_{0}=u^{\ast}.$
Now, if $b>0$ then%
\begin{align*}
x_{2}  &  =e^{-bx_{1}}x_{0}^{\lambda}e^{a_{1}-x_{0}}<x_{0}^{\lambda}%
e^{a_{1}-x_{0}}\leq(u^{\ast})^{\lambda}e^{a-u^{\ast}}=u^{\ast}\\
x_{3}  &  =e^{-bx_{2}}x_{1}^{\lambda}e^{a_{2}-x_{1}}<x_{1}^{\lambda}%
e^{a_{2}-x_{1}}\leq(u^{\ast})^{\lambda}e^{a-u^{\ast}}=u^{\ast}%
\end{align*}

Thus $x_{2},x_{3}\in(0,u^{\ast})$ and Lemma \ref{eoz} implies that
$\lim_{n\rightarrow\infty}(x_{n},y_{n})=(0,0);$ i.e. $(u^{\ast},u^{\ast
}/s^{\prime})\in E_{0}$ and the proof is complete.
\end{proof}

\medskip

\begin{remark}
A comparison of Corollaries \ref{nic}(b),(c) and \ref{eozc} indicates that the
base component $E_{0}$ of the extinction region is enlarged when $b>0$ as
compared with $b=0.$ In fact, the boundary point $(u^{\ast},u^{\ast}%
/s^{\prime})$ in the proof of Corollary \ref{eozc} is a fixed point of the
system when $b=0.$

\medskip
\end{remark}

The next result ensures that certain parameter ranges that appear in the
theorem that follows it are not empty.

\begin{lemma}
\label{lam}Assume that $\lambda>1$. Then%
\begin{equation}
(\lambda-1)[1-\ln(\lambda-1)]<\lambda-(\lambda-1)\ln\lambda\label{lam1}%
\end{equation}

Further,%
\begin{equation}
(\lambda-1)[1-\ln(\lambda-1)+\ln(b+1)]\leq\lambda-(\lambda-1)\ln
\lambda\label{lam2}%
\end{equation}

if and only if%
\begin{equation}
b\leq\frac{\lambda-1}{\lambda}e^{1/(\lambda-1)}-1. \label{lam4}%
\end{equation}

\end{lemma}

\begin{proof}
The inequality in (\ref{lam1}) is equivalent to%
\begin{equation}
\ln\frac{\lambda}{\lambda-1}<\frac{1}{\lambda-1} \label{lam3}%
\end{equation}

If we define $\gamma=1/(\lambda-1)$ then $\lambda=1+1/\gamma$ so (\ref{lam3})
is equivalent to $\ln(\gamma+1)<\gamma$, or $\gamma+1<e^{\gamma}$ which is
clearly true for $\gamma>0$. Further, the inequality in (\ref{lam2}) is
equivalent to
\[
b+1\leq\frac{e^{\gamma}}{\gamma+1}%
\]
which is equivalent to (\ref{lam4}).
\end{proof}

\medskip

\begin{theorem}
\label{nfpa}Assume that $\lambda>1.$

(a) Every non-negative solution of (\ref{4ab}) converges to 0 if and only if
(\ref{ac0}) holds.

(b) If
\begin{equation}
(\lambda-1)[1-\ln(\lambda-1)]<a<(\lambda-1)[1-\ln(\lambda-1)+\ln(b+1)]
\label{fxp0}%
\end{equation}
then (\ref{4ab}) has no positive fixed points but it has positive solutions
that do not converge to 0.

(c) Assume that (\ref{lam4}) holds and further,%
\begin{equation}
(\lambda-1)[1-\ln(\lambda-1)+\ln(b+1)]\leq a\leq\lambda-(\lambda-1)\ln\lambda.
\label{fxp1}%
\end{equation}

If $x^{\ast}$ is the smaller of the positive fixed points of (\ref{4ab}) then
there are initial values $x_{1},x_{0}\in(0,x^{\ast})$ for which the
corresponding positive solutions of (\ref{4ab}) do not converge to 0.

(d) If (\ref{lam4}) and (\ref{fxp1}) hold, then the solutions of (\ref{4ab})
from initial values $(x_{0},x_{1})\in([x^{\ast},\lambda]\times\lbrack
0,x^{\ast}])\cup([0,x^{\ast}]\times\lbrack x^{\ast},\lambda])$ do not converge
to the origin.
\end{theorem}

\begin{proof}
(a) If (\ref{ac0}) holds then by Theorem \ref{bc0} every positive solution of
(\ref{4ab}) converges to $0.$ Conversely, assume that (\ref{nac0}) holds. If
we choose $x_{0}=0$ in (\ref{4ab}) then $x_{2n}=0$ for $n\geq0$ so
\begin{equation}
x_{2n+1}=x_{2n-1}^{\lambda}e^{a-x_{2n-1}} \label{o1o}%
\end{equation}

The inequality in (\ref{nac0}) implies that $u^{\ast}$ is a fixed point of
(\ref{o1o}). If $x_{1}=u^{\ast}$ then the constant solution $x_{2n+1}=u^{\ast
}$ satisfies (\ref{o1o}) and thus the sequence $\{u^{\ast},0,u^{\ast}%
,0,\ldots\}$ with period 2 is a non-negative solution of (\ref{4ab}) that does
not converge to 0.

(b) If (\ref{fxp0})\ holds then Lemma \ref{2fp} implies that (\ref{4ab}) has
no positive fixed points. Further, by Lemma \ref{lam} we may choose a value of
$a$ such that
\[
(\lambda-1)[1-\ln(\lambda-1)]<a<\min\{(\lambda-1)[1-\ln(\lambda-1)+\ln
(b+1)],\lambda-(\lambda-1)\ln\lambda\}.
\]

For the above range of values of $a$, Lemma \ref{ord1} implies that the map
$f$ defined by (\ref{fmap}) has a pair of fixed points $u^{\ast}$ and $\bar
{u}$, that the interval $(u^{\ast},\bar{u})$ is invariant under $f$ with
$\bar{u}<f(\lambda)\leq\lambda$ and further, $f$ is strictly increasing on
this interval with $f(u)>u$ for $u\in(u^{\ast},\bar{u})$. Thus, if $u^{\ast
}<x_{0}<\bar{u}$ then $u^{\ast}<x_{0}<f(x_{0})<\bar{u}$. Let $\varepsilon
_{0}>0$ be fixed and choose $x_{1}\leq e^{-(a+\varepsilon_{0})/(\lambda-1)}$.
Since%
\begin{equation}
x_{2}=x_{0}^{\lambda}e^{a-bx_{1}-x_{0}}=e^{-bx_{1}}f(x_{0}) \label{ae1}%
\end{equation}
it follows that $x_{2}<f(x_{0})<\bar{u}.$ Now, let $\delta=e^{-\varepsilon
_{0}}\in(0,1)$ and choose $x_{1}$ small enough that%
\begin{equation}
e^{-bx_{1}/(1-\delta)}f(x_{0})>u^{\ast}. \label{ae2}%
\end{equation}

Then $e^{-bx_{1}}f(x_{0})>u^{\ast}$ in which case $u^{\ast}<x_{2}<\bar{u}.$
Proceeding in this fashion to the subsequent steps, assume by way of induction
that $u^{\ast}<x_{2n+1}<\bar{u}$ for some $n\geq1.$ Then $u^{\ast}%
<f(x_{2n+1})$ and as in the proof of Theorem \ref{bc0}(b), $x_{2n}\leq\delta
x_{2n-2}<\delta^{n}x_{0}$. Next,%
\begin{equation}
x_{2n+3}=x_{2n+1}^{\lambda}e^{a-bx_{2n+2}-x_{2n+1}}=e^{-bx_{2n+2}}f(x_{2n+1})
\label{ae3}%
\end{equation}
so $x_{2n+3}<f(x_{2n+1})<f(\bar{u})=\bar{u}$ since $f$ \ is increasing.
Further, by (\ref{ae3}) and the fact that $f(u)>u$ for $u\in(u^{\ast},\bar
{u}),$
\[
x_{2n+3}=e^{-bx_{2n+2}}f(e^{-bx_{2n}}f(x_{2n-1}))>e^{-bx_{2n+2}-bx_{2n}%
}f(x_{2n-1})>\cdots>e^{-bx_{2n+2}-bx_{2n}-\cdots-bx_{2}}f(x_{1})
\]

As in Theorem \ref{bc0}(b)
\[
bx_{2n+2}+bx_{2n}+\cdots+bx_{2}\leq bx_{0}(1+\delta+\cdots+\delta^{n+1}%
)<\frac{bx_{0}}{1-\delta}%
\]
so by (\ref{ae2}) $x_{2n+3}>u^{\ast}.$ By induction, this inequaltiy holds for
all $n\geq1$ so that this positive solution $\{x_{n}\}$ of (\ref{4a}) does not
converge to 0.

(c) If (\ref{lam4}) holds then (\ref{fxp1}) defines a nonempty interval and by
Lemma \ref{2fp}, (\ref{4ab}) has at least one fixed point $x^{\ast}$. Since
(\ref{fxp1}) also implies $a>(\lambda-1)[1-\ln(\lambda-1)]$, by Lemma
\ref{ord1} the map $f$ defined by (\ref{fmap}) also has a pair of fixed points
$u^{\ast}$ and $\bar{u}$ and Lemma \ref{bux} implies that $u^{\ast}<x^{\ast}$.
Let $u^{\ast}<x_{-1}<\min\{x^{\ast},\bar{u}\}$ and repeat the proof of (b) to
obtain a positive solution $\{x_{n}\}$ that does not converge to 0.

(d) By (\ref{lam4}) and (\ref{fxp1}), (\ref{4ab}) has a fixed point $x^{*}%
\leq\frac{\lambda-1}{b+1}<\lambda$, and $f(\lambda)\leq\lambda$, where
$\lambda$ is the unique maximum of the map $f$ defined in (\ref{fmap}).
Moreover, for $x, y\geq0$
\[
x^{\lambda}e^{a-by-x}=f(x)e^{-by}\leq f(x)\leq f(\lambda)
\]
so the solutions of (\ref{4ab}) are bounded from above by $f(\lambda)$. Now,
for $n\geq0$, if $x^{*}\leq x_{2n-1}\leq\lambda, x_{2n}\leq x^{*}$
\[
x_{2n+1}=x_{2n-1}^{\lambda}e^{a-bx_{2n}-x_{2n-1}}=f(x_{2n-1})e^{-bx_{2n}}\geq
f(x^{*})e^{-bx^{*}}=x^{*}
\]
and
\[
x_{2n+2}=x_{2n}^{\lambda}e^{a-bx_{2n+1}-x_{2n}}=f(x_{2n})e^{-bx_{2n+1}}\leq
f(x^{*})e^{-bx^{*}}=x^{*}
\]
Similarly, for $n\geq0$, if $x_{2n-1}\leq x^{*}, x^{*}\leq x_{2n}\leq\lambda$%

\[
x_{2n+1}=x_{2n-1}^{\lambda}e^{a-bx_{2n}-x_{2n-1}}=f(x_{2n-1})e^{-bx_{2n}}\leq
f(x^{*})e^{-bx^{*}}=x^{*}
\]
and
\[
x_{2n+2}=x_{2n}^{\lambda}e^{a-bx_{2n+1}-x_{2n}}=f(x_{2n})e^{-bx_{2n+1}}\geq
f(x^{*})e^{-bx^{*}}=x^{*}
\]
and the proof is complete.
\end{proof}

\medskip

The following result about the system (\ref{1})-(\ref{2}) is an immediate
consequence of Theorem \ref{nfpa}. It may be compared with Corollaries
\ref{nic} and \ref{eozc}.

\begin{corollary}
\label{ext}In the system (\ref{1})-(\ref{2}) assume that $\lambda>1$, $b>0$,
$s_{n}=0$ for all $n$ and further, $s_{n}^{\prime}=s^{\prime}$ and $r_{n}=r$
are constants.

(a) The origin is a global attractor of all orbits $\{(x_{n},y_{n})\}$ of
(\ref{1})-(\ref{2}) if and only if (\ref{ac0}) holds.

(b) If $a=r+\ln s^{\prime}$ satisfies (\ref{fxp0}) then the origin is the only
fixed point of (\ref{1})-(\ref{2}) but the system has orbits in the positive
quadrant of the plane that do not converge to the origin, i.e. $E_{0}%
\not =[0,\infty)\times\lbrack0,\infty)$.

(c) If $b$ and $a=r+\ln s^{\prime}$ satisfy the inequalities in (\ref{lam4})
and (\ref{fxp1}) respectively, then there are initial points $(x_{0},y_{0}%
)\in\lbrack0,x^{\ast})\times\lbrack0,x^{\ast}/s^{\prime})$ for which the
corresponding orbit does not converge to the origin, i.e. $[0,x^{\ast}%
)\times\lbrack0,x^{\ast}/s^{\prime})\not \subset E_{0}.$

(d) If $b$ and $a=r+\ln s^{\prime}$ satisfy the inequalities in (\ref{lam4})
and (\ref{fxp1}) respectively, then there are initial points $(x_{0},y_{0}%
)\in[x^{*}, \lambda]\times[0, x^{*}/s^{\prime}]\cup[0, x^{*}]\times
[x^{*}/s^{\prime}, \lambda/s^{\prime}]$ for which the corresponding orbit does
not converge to the origin.
\end{corollary}

By Corollary \ref{eozc} $[0,u^{\ast})\times\lbrack0,u^{\ast}/s^{\prime
})\subset E_{0}.$ On the other hand, Corollary \ref{ext}(c) indicates that the
Allee region $E_{0}$ does not contain the larger rectangle $[0,x^{\ast}%
)\times\lbrack0,x^{\ast}/s^{\prime})$ (when $x^{\ast}$ exists). Actually, the
following is true.

\begin{corollary}
Under the hypotheses of Corollary \ref{ext}(c) $E_{0}\subset\lbrack0,x^{\ast
})\times\lbrack0,x^{\ast}/s^{\prime})$ where the inclusion is proper.
\end{corollary}

\begin{proof}
By Corollary \ref{ext}(d), the boundary of the rectangle $[0,x^{\ast}%
)\times\lbrack0,x^{\ast}/s^{\prime})$ is not contained in $E_{0}$ and since
$E_{0}$ is a connected set, it cannot contain points outside $[0,x^{\ast
})\times\lbrack0,x^{\ast}/s^{\prime})$.
\end{proof}

\medskip

\begin{figure}[tbp] 
  \centering
  \includegraphics[width=3.83in,height=3.09in,keepaspectratio]{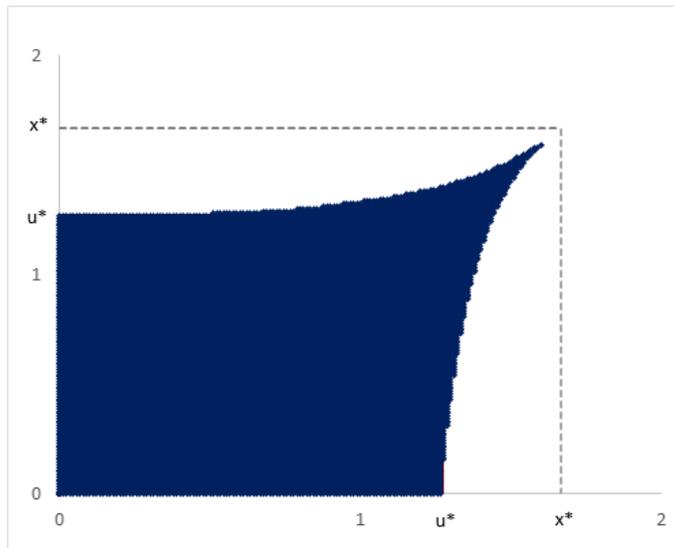}
  \caption{$E_{0}$ with $\lambda = 3$, $a=0.7936$, $b=0.0891$, $s'=1$}
  \label{E0}
\end{figure}

The numerically generated plot of $E_{0}$ shown in Figure \ref{E0} clearly
illustrates the preceding result. This plot was generated by examining the
behavior of solutions from initial points $(x_{0},y_{0})$ on a 200 by 200
partition grid. As expected from the adverse effect of inter-stage interaction
with $b>0$, the extinction region is larger then when $b=0.$ However, not all
initial pairs $(x_{0},y_{0})$ in the set $[0,x^{\ast})\times\lbrack0,x^{\ast
}/s^{\prime})$ lead to extinction, a somewhat non-intuitive outcome.

\section{ Summary and open problems}

In this paper, we established general results about the convergence to origin
of orbits of the system (\ref{1})-(\ref{2}) in the positive quadrant of the
plane when $\lambda>1$. These extinction results that involve time-dependent
parameters, are in line with expectations about the behavior of the orbits of
the system.

We also studied the special case where $s_{n}=0$ for all $n$ in greater detail
and determined that while the existence of fixed points in the positive
quadrant is a sufficient condition for survival, it is not necessary. This
surprising fact is true when $b>0,$ i.e. when the stages (adults and
juveniles) interact within each period $n$ but false if $b=0$ and inter-stage
interactions do not occur, a case that includes first-order population models
where there are no stages or age-structuring. Also non-intuitively, we found
that although the Allee equilibrium moves away from the origin due to
interactions between stages and leads to an enlargement of the extinction
region, this enlargement is not the maximum possible allowed by the shift.

There are many open questions about the nature of the extinction region $E$,
and its complement, the survival region. We pose a few of these questions as
open problems. The first concerns the base component $E_{0}$ of $E$.

\begin{problem}
\label{p1} Determine $E_{0}$ and its boundary, namely the extinction or Allee
threshold, under the hypotheses of Corollary \ref{ext}(c).
\end{problem}

\medskip

The numerically generated image in Figure \ref{E0} suggests that Problem
\ref{p1} has a nontrivial solution. This problem naturally leads to the following.

\begin{problem}
\label{p2} Determine the extinction set $E$ under the hypotheses of Corollary
\ref{ext}(c).
\end{problem}

\begin{figure}[tbp] 
  \centering
  \includegraphics[width=4.15in,height=3.09in,keepaspectratio]{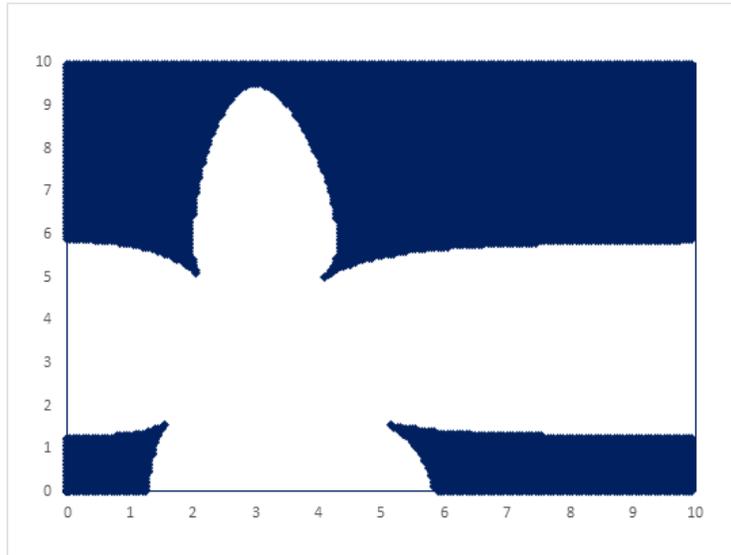}
  \caption{$E$ (shaded) and its complement for $\lambda = 3$, $a=0.7936$, 
$b=0.0891$, $s'=1$}
  \label{E}
\end{figure}

Figure \ref{E} illustrates a numerically generated part of $E$. This figure
shows 3 distinct components, two of which are unbounded. The following result
verifies that $E$ must have unbounded components.

\begin{proposition}
Assume that $\lambda>1$ and (\ref{fup}) holds. Let $u_{\ast}$ as defined in
Lemma \ref{ord1}, be the point such that $f(u_{\ast})=u^{\ast}$. If
\[
(x_{0},x_{1})\in R_{0,1}\cup R_{1,0}\cup R_{1,1}%
\]
where the rectangles $R_{0,1},R_{1,0},R_{1,1}$ are defined as%
\[
R_{0,1}=[0,u^{\ast})\times\lbrack u_{\ast},\infty),\quad R_{1,0}=[u_{\ast
},\infty)\times\lbrack0,u^{\ast}),\quad R_{1,1}=[u_{\ast},\infty)\times\lbrack
u_{\ast},\infty)
\]
then corresponding orbits of (\ref{4ab}) converge to zero, i.e. the above set
is contained in the extinction region $E$.
\end{proposition}

\begin{proof}
For $i=0,1$, if $x_{i}<u^{\ast}$, then by Lemma \ref{eoz}, $\lim
_{k\rightarrow\infty}x_{i+2k}=0$. If $x_{i}\geq u_{\ast}$, then
\[
x_{i+2}=x_{i}^{\lambda}e^{a-bx_{i+1}-x_{i}}=f(x_{i})e^{-bx_{i+1}}\leq u^{\ast
}e^{-bx_{i+1}}<u^{\ast}%
\]
and the rest is again a consequence of Lemma \ref{eoz}.
\end{proof}

\medskip

We also saw that even when (\ref{4ab}) has no positive fixed points if
(\ref{fxp0}) holds, the set $E$ is not the entire quadrant $[0,\infty)^{2}$.
This surprising fact is illustrated in the numerically generated panels in
Figure \ref{nx2}. We also see in this figure that $E$ has fewer distinct
components for the larger value of $b$ and that the survival region (unshaded)
gets disconnected when the components of $E$ join.

\begin{figure}[tbp] 
  \centering
  \includegraphics[width=5.96in,height=2.76in,keepaspectratio]{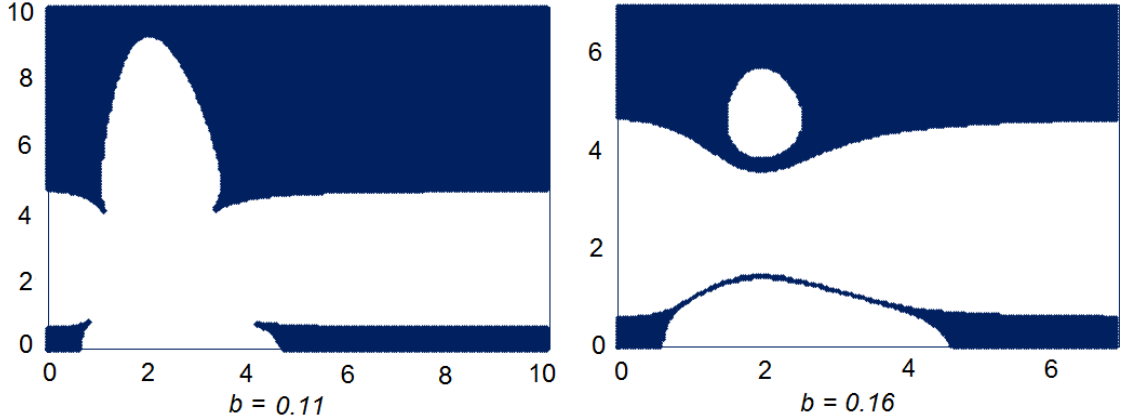}
  \caption{$E$ for $\lambda = 2$, $a=1.1$, $s'=1$ and two different values 
of $b$}
  \label{nx2}
\end{figure}

These figures motivate the following.

\begin{problem}
Determine the extinction set $E$, or more to the point, its complement the
survival set, when (\ref{fxp0}) holds.
\end{problem}

\medskip Settling the following conjecture may be relevant to the preceding study.

\begin{conjecture}
Assume that (\ref{fxp0}) holds. Then for every positive solution $\{ x_{n}\}$
of (\ref{4ab}) at least one of the two subsequences $\{ x_{2n}\}$ or $\{
x_{2n-1}\}$ converges to zero.
\end{conjecture}

\medskip

Another direction to pursue involves extending the range of the parameter $a$.

\begin{problem}
Explore the extinction and survival regions if $a>\lambda-(\lambda
-1)\ln\lambda$.
\end{problem}

Finally, it may be appropriate to close with the following.

\begin{problem}
Extend the preceding analysis to the more general equation (\ref{4a}) with
$s>0$ and $b>0.$
\end{problem}

\end{document}